\documentclass{article}
\usepackage{tikz,amsmath,amsthm,verbatim,enumitem} 
\setlist{itemsep=0ex}
\usetikzlibrary{matrix,arrows,decorations.pathreplacing,decorations.pathmorphing,positioning}
\tikzset{baseline=(current bounding box.west)}
\tikzset{every matrix/.style={
execute at begin cell=\node\bgroup$,
execute at end cell=$\egroup;%
,minimum size=4mm,
matrix anchor=south west,
inner sep=0pt
}}
\definecolor{Dark}{rgb}{0,0.4,0.2}  
\newcommand{\Dark}[1]{\textcolor{Dark}{#1}}
\usepackage[T1]{fontenc}
\usepackage[bitstream-charter]{mathdesign}

\usepackage[latin1]{inputenc}		

\usepackage{todonotes}

\newtheorem{theorem}{Theorem}
\newtheorem{lemma}[theorem]{Lemma}
\newtheorem{corollary}[theorem]{Corollary}
\newtheorem{conjecture}[theorem]{Conjecture}
\theoremstyle{definition}

\newtheorem{example}[theorem]{Example}

\usepackage{sectsty}
\usepackage[compact]{titlesec} 
\allsectionsfont{\color{Dark}\scshape\selectfont}

\makeatletter							
\def\printtitle{
    {\color{Dark} \centering \huge \sc \textbf{\@title}\par}}		
\makeatother							

\title{Conditions To Extend Partial Latin Rectangles}

\makeatletter							
\def\printauthor{
    {\centering \small \@author}}				
\makeatother							

\author{%
	Serge Ballif \\
	ballif@math.psu.edu \\
	\vspace{20pt}
	}

\usepackage[runin]{abstract}			
\abslabeldelim{\quad}						%
\begin{document} 
\printtitle 

\printauthor

\begin{abstract}
In 1974 Allan Cruse provided necessary and sufficient conditions to extend an $r\times s$ partial latin rectangle consisting of $t$ distinct symbols to a latin square of order $n$. Here we provide some generalizations and consequences of this result. Our results are obtained via an alternative proof
of Cruse's theorem.
\end{abstract}

\section{Introduction}
The question of whether a partial latin square can be completed to a latin square of the same order is known is known to be NP-complete \cite{Col84}. Yet, there are conditions that can guarantee when a given partial latin square is completable to a latin square. The most well-known result is the Evans Conjecture, proved in \cite{Sme81} and \cite{AH83}, that any partial latin square of order $n$ with at most $n-1$ entries can be completed to a latin square. The following two theorems are also well known, and they apply to problems more often encountered in practice. 

\begin{theorem}[Ryser's Theorem \cite{Rys51}]
An $r\times s$ latin rectangle, $R$, consisting of the symbols $1,2,\ldots,n$ can be extended to a latin square of order $n$ if and only if each of the $n$ symbols occurs at least $r+s-n$ times inside $R$.
\end{theorem}

\begin{theorem}[Evans's Theorem \cite{Eva60}]
A partial latin square of order $r$ can be extended to (embedded inside) a latin square of order $n$ for each $n\ge 2r$.
\end{theorem}

The condition $n\ge 2r$ of Evan's Theorem is the best possible sufficient condition to guarantee that any partial latin square of order $r$ can be embedded inside a square of order $n$. However, many partial latin squares of order $r$ can be extended to a latin square of order $<2r$. In \cite{Cru74} Allan Cruse simultaneously generalized both Ryser's Theorem and Evan's Theorem with a beautiful theorem that provided necessary and sufficient conditions to guarantee when a partially completed latin rectangle could be extended to a latin square.

\begin{theorem}[Cruse's Theorem]
Let $r,s,t\le n$. An $r\times s$ partial latin rectangle, $R$, consisting of $t$ distinct symbols can be completed to a latin square of order $n$ if and only if it can be extended (using the same $t$ symbols) to an $r\times s$ partial latin rectangle, $P$, such that the following four conditions hold. 
\begin{enumerate}[label=\textup{(A\arabic*)},ref=A\arabic*]
\item Each row of $P$ has at least $s+t-n$ entries. \label{rows}
\item Each column $P$ has at least $r+t-n$ entries. \label{columns}
\item Each of the $t$ symbols occurs at least $r+s-n$ times in $P$. \label{symbols}
\item The number of entries in $P$ does not exceed \label{entries}
\[
\frac{rst+(n-r)(n-s)(n-t)}{n}.
\]
\end{enumerate}
\end{theorem} 
Ryser's Theorem is precisely the special case of Cruse's Theorem with $t=n$, and Evans's Theorem is the special case of Cruse's Theorem with $r=s=t\le \frac{n}{2}$.  

In this paper we shall prove two generalizations of Cruse's Theorem for structures related to latin squares. In Section~\ref{Frequency Cruse Theorem} we present a generalization of Cruse's Theorem that gives necessary and sufficient conditions to complete a partial frequency rectangle to a frequency square. In Section~\ref{Saturated Rectangle Theorem} we exhibit a different generalization of Cruse's Theorem related to completing partial latin rectangles to a partial latin rectangle with a maximum number of filled cells for a given number of rows, columns, and symbols. 

We also provide a proof of Cruse's Theorem that is perhaps easier to visualize than the proof provided by Cruse in  \cite{Cru74}. This alternative proof is outlined in Section~\ref{Proof of Cruse's Theorem}, and two key lemmas in the proof are proved in Sections~\ref{Shuffle Lemma} and \ref{Gap Filling Lemma}.  These lemmas will be the main steps in the alternative proof of Cruse's Theorem and in the proofs given in Sections 6 and 7 of two theorems (Theorems~\ref{frequency} and \ref{saturated}) that generalize it.

In Section~\ref{k-plex} we introduce a generalization of a transversal of a latin square called a partial $k$-plex, and we demonstrate the relationship between partial $k$-plexes and quasi-embeddings of latin squares inside latin squares of larger order. We end with a generalization of Brualdi's Conjecture that every latin square of order $n$ has a partial transversal of size $n-1$.

\section{Generalizations of Cruse's Theorem}\label{Generalizations of Cruse's Theorem}
\subsection{Definitions}

A \emph{latin square of order $n$} is an $n\times n$ array filled with $n$ distinct symbols, each occurring exactly once in each row and exactly once in each column. For $m\le n$ an \emph{$m\times n$ latin rectangle} is an $m\times n$ array filled with $n$ distinct symbols each occuring once in each row and at most once in each column. It will be convenient to assume that the symbols are $1,2,\ldots,n$.

A \emph{partial latin square} of order $n$ is an $n\times n$ array (possibly with empty cells) based on $n$ distinct symbols such that each row and column contains each of the $n$ symbols at most once. Latin squares are special cases of partial latin squares. An \emph{extension} of a partial latin square $P$ is a (partial) latin square $P'$ such that the such that both $P$ and $P'$ share the same entry in row $i$ and column $j$ whenever that cell is nonempty in $P$. 

The entries of a partial latin square can be encoded in a set of ordered triples, $(r,c,s)$, containing the information (row, column, symbol). For each partial latin square $P=\{(i,j,k)\}$ and each permutation $\sigma$ of the triple $(r,c,s)$ we can define a \emph{conjugate} of $P$ to be the partial latin square $\sigma(P)=\{\big(\sigma(i),\sigma(j),\sigma(k)\big)\}$. 

Observe that the symmetry of conditions \eqref{rows}--\eqref{entries} in Cruse's Theorem is consistent with the fact that that $R$ can be completed to an $n\times n$ latin square if and only if its conjugates can be completed to an $n\times n$ latin square.

\subsection{Completing Latin Frequency Rectangles}\label{Frequency Cruse Theorem}
A Latin squares is an example of a frequency square. For $n=\lambda_1+\cdots+\lambda_k$ we define a \emph{frequency square} or F-\emph{square} on the set of symbols $\{x_1,\ldots,x_k\}$ to be an $n\times n$ array where the symbol $x_i$ occurs precisely $\lambda_i$ times in each row and column.  We say that such a square is of type $F(n;\lambda_1,\ldots,\lambda_k)$. It is usually convenient to take $x_i=i$.

A \emph{partial} F\emph{-square} is an $n\times n$ array where the symbol $x_i$ occurs at most $\lambda_i$ times in each row and column. A \emph{partial} F-\emph{rectangle} is defined analogously. We say that an $r\times s$ partial F-rectangle is of type $F(t,\lambda_1,\ldots,\lambda_k)$ if $t=\lambda_1+\cdots+\lambda_k$ where $\lambda_i$ is the smallest number such that each symbol $x_i$ occurs at most $\lambda_i$ times in each row and column. 

\begin{theorem}\label{frequency}
Let $\lambda_1+\cdots+\lambda_k=n$ and $\mu_1+\cdots+\mu_k=t$ be partitions satisfying $0\le \mu_i\le \lambda_i$. An $r\times s$ partial F-rectangle $R$ of type $F(t,\mu_1,\ldots,\mu_k)$ can be extended to an $n\times n$ $F$-square of type $F(n,\lambda_1,\ldots,\lambda_k)$ if and only if $R$ can be extended to a partial F-rectangle, $R'$, also of type $F(t,\mu_1,\ldots,\mu_k)$ such that the following four conditions hold.
\begin{enumerate}[label=\textup{(B\arabic*)},ref=B\arabic*]
\item Each row of $R'$ has at least $s+t-n$ entries. \label{F-rows}
\item Each column $R'$ has at least $r+t-n$ entries. \label{F-columns}
\item The symbol $x_i$ occurs at least $\mu_i(r+s-n)$ times in $R'$. \label{F-symbols}
\item The number of entries in $R'$ does not exceed \label{F-entries}
\[
\frac{rst+(n-r)(n-s)(n-t)}{n}.
\]
\end{enumerate}
\end{theorem}

As an example, below is a rectangle of type $F(4;2,2)$ with $r=3$ and $s=t=4$. We check the conditions with $n=5$. There are at least $4+4-5=3$ entries in each row and at least $3+4-5=2$ entries in each column. The symbols 1 and 2 each occur at least $2(3+4-5)=4$ times and the number of entries is not more than $((3)(4)(4)+(2)(1)(1))/5=10$. Therefore, the conditions \eqref{F-rows}--\eqref{F-entries} are satisfied, so there exists an extension to a $5\times 5$ F-square of any type compatible with a type $F(4;2,2)$ partial subrectangle. We exhibit two such extensions.
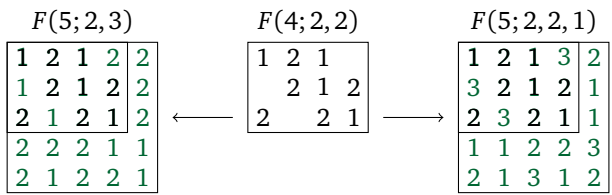
\[
\begin{tikzpicture}[scale=.4]
\begin{scope}
\draw[fill=white] (0,5) rectangle (4,2);
\matrix[yshift=.8cm]
{ 
1&2&1&&\\
&2&1&2&\\
2&&2&1&\\
\\
\\
};
\draw[->] (4.5,2.5) -- (6.5,2.5);
\draw[->] (-.5,2.5) -- (-2.5,2.5);
\draw (2,5) node[above]{$F(4;2,2)$};
\end{scope}
\begin{scope}[xshift=7cm]
\draw[fill=white] (0,0) rectangle(5,5);
\draw (0,5) rectangle (4,2);
\matrix[color=Dark] 
{ 
1&2&1&3&2\\
3&2&1&2&1\\
2&3&2&1&1\\
1&1&2&2&3\\
2&1&3&1&2\\
};
\matrix[yshift=.8cm]
{ 
1&2&1&&\\
&2&1&2&\\
2&&2&1&\\
&&&&\\
&&&&\\
};
\draw (2.5,5) node[above]{$F(5;2,2,1)$};
\end{scope}
\begin{scope}[xshift=-8cm]
\draw[fill=white] (0,0) rectangle(5,5);
\draw (0,5) rectangle (4,2);
\matrix[color=Dark]
{ 
1&2&1&2&2\\
1&2&1&2&2\\
2&1&2&1&2\\
2&2&2&1&1\\
2&1&2&2&1\\
};
\matrix[yshift=.8cm]
{ 
1&2&1&&\\
&2&1&2&\\
2&&2&1&\\
&&&&\\
&&&&\\
};
\draw (2.5,5) node[above]{$F(5;2,3)$};
\end{scope}
\end{tikzpicture}
\]

\subsection{Saturated Rectangles of Type $(r,s,t)$}\label{Saturated Rectangle Theorem}
In this section we introduce a natural generalization of Cruse's Theorem that gives conditions under which a partial latin rectangle can be extended to a partial latin rectangle that is filled with the maximum possible number of symbols.

A \emph{partial latin rectangle of type $(r,s,t)$} is an $r\times s$ array with cells that are either empty or contain entries from a set of $t$ elements so that no symbol occurs more than once in any row or column. A partial latin rectangle of type $(r,s,t)$ is a direct generalization of several latin structures. Any $r\times s$ subrectangle of a partial latin square is a latin rectangle of type $(r,s,t)$ for some $t$. A partial latin rectangle of type $(m,n,n)$ is an $m\times n$ latin rectangle, and a partial latin rectangle of type $(n,n,n)$ is a partial latin square of order $n$. If $P$ is a partial latin rectangle of type $(r,s,t)$ then its conjugate, $\sigma(P)$, is a partial latin rectangle of type $(\sigma(r),\sigma(s),\sigma(t))$.

The maximum possible number of entries in a partial latin rectangle of type $(r,s,t)$ is $\min\{rs,rt,st\}$. To see this bound, we note that conjugation does not change the number of entries. Thus, we can assume without loss of generality, that $t\ge r,s$, so there only $rs$ cells in the rectangle.  

We say that a partial latin rectangle of type $(r,s,t)$ is \emph{saturated} if it contains the maximum possible number of entries, namely
\[
\min\{rs,rt,st\}.
\]

In a sense, a saturated partial latin rectangle of type $(r,s,t)$ is the best we can do since it is as complete as possible for the triple $(r,s,t)$. A partial latin rectangle of type $(n,n,n)$ is a latin square if and only if it is saturated.

The saturation property of a partial latin rectangle of type $(r,s,t)$ is preserved in each of its conjugate rectangles. For example, we can let $\sigma$ swap columns and symbols. Below we display two partial latin rectangles of type $(4,5,4)$, $P$ and $Q$, and their conjugates $\sigma(P)$ and $\sigma(Q)$.
\[
\begin{tikzpicture}[scale=.4]
\begin{scope}
\draw[fill=white] (0,0) grid (5,4);
\draw (2.5,4) node[above]{$P$};
\matrix
{
1&2&3&4&\\
2&4&&3&1\\ 
&3&4&1&2\\
4&&1&2&3\\
};
\end{scope}
\begin{scope}[xshift=7cm]
\draw[fill=white] (0,0) grid (4,4);
\draw (2,4) node[above]{$\sigma(P)$};
\matrix
{
1&2&3&4\\
5&1&4&2\\
4&5&2&3\\
3&4&5&1\\
};
\end{scope}
\begin{scope}[xshift=13cm]
\draw[fill=white] (0,0) grid (5,4);
\draw (2.5,4) node[above]{$Q$};
\matrix
{
1&2&&&3\\
2&1&&&4\\
&&3&4&2\\
&&4&3&1\\
};
\end{scope} 
\begin{scope}[xshift=20cm]
\draw[fill=white] (0,0) grid (4,4);
\draw (2,4) node[above]{$\sigma(Q)$};
\matrix
{
1&2&5&\\
2&1&&5\\
&5&3&4\\
5&&4&3\\
};
\end{scope}
\end{tikzpicture}
\]
Both $P$ and $\sigma(P)$ are saturated partial latin rectangles. However $Q$ and $\sigma(Q)$ are not saturated, although they are maximal in the sense that no further entries can be added without increasing $r$, $s$, or $t$.

Our next definition is derived from terms in the reference \cite{BBJ09} which is a paper that generalizes Hall's Marriage Theorem. Let $\mathcal{A}$ be a collection of sets $A_1,\ldots,A_r$, and let $S=A_1\cup A_2\cup\cdots\cup A_r$.  Let $f\colon S\to \mathbb{N}$ and $g\colon\mathcal{A}\to\mathbb{N}$ be functions. We say that $\mathcal{A}$ has a \emph{system of $f,g$ representatives of $\mathcal{A}$}  if to every set $A_j\in \mathcal{A}$, we associate $g(A_j)$ representatives from $S$, and every vertex $a_i\in S$ is a representative of $f(a_i)$ sets from $\mathcal{A}$.

Cruse's Theorem provides necessary and sufficient conditions to complete a partial latin rectangle of type $(r,s,t)$ to an order $n$ latin square. It is much more difficult to ascertain whether a given partial latin rectangle of type $(r,s,t)$ can be extended to a saturated partial latin rectangle of type $(R,S,T)$ ($r\le R, s\le S, t\le T$). In the special case $R=S=T=n$, the answer is provided by Cruse's Theorem. 

\begin{theorem}
Let $A$ be a partial latin rectangle that can be completed to a saturated latin rectangle of type $(r,s,t)$, where $r\le s\le t$. Then $A$ can be completed to a saturated latin rectangle of type $(R,S,T)$ for each $R\le r$, $S\le s$, $T\ge t$.
\end{theorem}
\begin{proof}
A saturated latin rectangle of type $(r,s,t)$ must have all cells filled since $r\le s\le t$. Deleting rows or columns or increasing the number of available symbols still leaves all the cells filled. The condition $R,S\le T$ then guarantees that the rectangle is saturated.
\end{proof}

The four conditions of Cruse's theorem are sufficient conditions to guarantee that a partial latin rectangle of type $(r,s,t)$ can be extended to an saturated latin rectangle of type $(R,S,T)$ whenever at most one of $R$, $S$, and $T$ is $>n$. However, conditions \eqref{symbols} and \eqref{entries} are not necessary in general.

The following theorem gives necessary and sufficient conditions to guarantee that such an extension exists.

\begin{theorem}\label{saturated}
Let $r\le R$, $s\le S$, $t\le T$, and $R,S\le T$. An partial latin rectangle of type $(r,s,t)$ can be completed to an saturated partial latin rectangle of type $(R,S,T)$ if and only if it can be extended to a partial latin rectangle, $P$, of type $(R,s,t)$ such that the following four conditions hold. 
\begin{enumerate}[label=\textup{(C\arabic*)},ref=C\arabic*]
\item Each row of $P$ has at least $s+t-T$ entries. \label{Sat-rows}
\item Each column of $P$ has at least $R+t-T$ entries. \label{Sat-columns}
\item There exist functions $f$ and $g$ such that the collection $\mathcal{A}$ consisting of sets
\[A_i=\{1,2,\ldots,t\}\setminus\{\text{elements in row $i$ of $P$}\}\]
has a system of $f,g$-representatives where\label{Sat-symbols}
\begin{enumerate}[label=\textup{(C3\alph*)},ref=C3\alph*]
\item $f(i)\le S-s$ for each $1\le i\le t$. \label{Sat-a}
\item $g(A_j)\le S-s$ for each $1\le j\le R$ \label{Sat-b}
\item $g(A_j)\ge S-T+|A_j|$ for each $1\le j\le R$ \label{Sat-c}
\item $\sum_{j=1}^R g(A_j)=\sum_{i=1}^t f(i)\ge (S-s)(R+t-T)$.\label{Sat-d} 
\end{enumerate}
\end{enumerate}
\end{theorem}

If $R,S\not\le T$ then we can apply the theorem to one of the conjugates of the rectangle to obtain a saturated rectangle. Then reverse the conjugation.

\begin{example}
Consider the $(5,5,5)$-partial latin square $A$ below. We apply Cruse's Theorem to see that $A$ can be extended to a latin square of order $n\ge8$. The values of $(r,s,t)$ for which $A$  can be completed to an $(r,s,t)$-saturated  partial latin rectangle are precisely $(5,5,7)$, $(5,6,7)$, $(6,5,7)$, $(6,6,7)$ or when $r,s,t\ge5$ with at least one of $r\ge8$, $s\ge8$, or $t\ge8$. 
\[
\begin{tikzpicture}[scale=.4]
\begin{scope}
\draw[fill=white] (0,0) rectangle (5,5);
\matrix 
{
5&&3&4&2\\
&5&4&2&3\\
1&2&5\\
2&1&&5\\
&&2&3&4\\
};
\draw (2.5,0) node[below] {$A$};
\draw[->] (5.5,2.5) -- ++(1,0);
\end{scope}
\begin{scope}[xshift=7cm]
\draw[fill=white] (0,0) rectangle (5,5) rectangle (6,0) rectangle (5,-1) rectangle (0,0);
\matrix[yshift=-.4cm,color=Dark]
{
5&6&3&4&2&1\\
7&5&4&2&3&6\\
1&2&5&7&6&3\\
2&1&6&5&7&4\\
6&7&2&3&4&5\\
4&3&1&6&5&2\\
};
\matrix
{
5&&3&4&2\\
&5&4&2&3\\
1&2&5\\
2&1&&5\\
&&2&3&4\\
};
\end{scope}                                                                  
\end{tikzpicture}
\]
\end{example}

\section{Proof of Cruse's Theorem}\label{Proof of Cruse's Theorem}

Here we provide a proof that differs from the original proof of Cruse in \cite{Cru74}. Throughout the proof we shall assume the square is based on the symbols $1,\ldots, n$. We shall refer to the symbols $1,\ldots,t$ as the ``original $t$ symbols'' and the symbols $t+1,\ldots, n$ as the ``$n-t$ new symbols.'' We begin by showing the necessity of the conditions.

\begin{proof}[Proof of Necessity of \eqref{rows}--\eqref{entries}]
First we place an $r\times s$ partial latin square based on $t$ symbols in the upper left corner of and $n\times n$ array. Assume that $P$ can be completed to an $n\times n$ latin square without adding any of the original $t$ symbols to the original $r\times s$ rectangle. For convenience we draw such an $n\times n$ array below and provide the labels $A$ and $B$ to two of the empty regions of the $n\times n$ square.
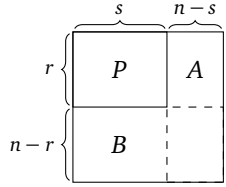
\[
\begin{tikzpicture}[scale=.25]
\draw[style=dashed] (5,4) rectangle (8,0);
\draw(0,0) rectangle (8,8);
\draw(0,8) rectangle (5,4);
\draw [decorate,decoration={brace,amplitude=3pt},yshift=5pt]
   (0.1,8)  -- (4.9,8) 
   node [black,midway,above=2pt] {\footnotesize $s$};
\draw [decorate,decoration={brace,amplitude=3pt},yshift=5pt]
   (5.1,8)  -- (7.9,8) 
   node [black,midway,above=2pt] {\footnotesize $n-s$};
\draw [decorate,decoration={brace,amplitude=3pt},xshift=-5pt]
   (0,4.1)  -- (0,7.9) 
   node [black,midway,left=2pt] {\footnotesize $r$};
\draw [decorate,decoration={brace,amplitude=3pt},xshift=-5pt]
   (0,.1)  -- (0,3.9) 
   node [black,midway,left=2pt] {\footnotesize $n-r$};
\draw (2.5,6) node{$P$};
\draw (6.5,6) node{$A$};
\draw (2.5,2) node{$B$};
\end{tikzpicture}
\]

To prove the necessity of condition \eqref{rows} we must show that each row of $P$ contains at least $s+t-n$ entries. There are at most $n-t$ of the new symbols to fill in the empty cells of each row. Since it is possible to fill the empty cells of each row of $P$ with new symbols there must be at least $s-(n-t)=s+t-n$ filled cells in each row. 

Condition \eqref{columns} is equivalent to condition \eqref{rows} by swapping rows and columns and reversing the roles of $r$ and $s$. Likewise, condition \eqref{symbols} is equivalent to condition \eqref{rows} by swapping symbols and rows and reversing the roles of $r$ and $t$. Alternatively, one can note the necessity of condition \eqref{symbols} by 
observing that the maximum number of times that any symbol can occur outside $P$ is $(n-s)+(n-r)=2n-s-r$. Thus if a symbol occurs $j$ times in $P$, then $j+(2n-s-r)\ge n$, so $j\ge s+r-n$.

To see the necessity of condition \eqref{entries} we consider the rectangle $A$. Each column of $A$ requires $r$ entries, at most $n-t$ of which are from the set of $n-t$ new symbols. Thus each of the $n-s$ columns of $A$ will require at least $r-(n-t)$ original symbols, so $A$ will consist of at least $(n-s)(r+t-n)$ original symbols. Each of the $r$ rows requires $t$ original symbols, so $P$ can contain at most
\[
rt-(n-s)(r+t-n)=\frac{rst+(n-r)(n-s)(n-t)}{n}
\]
of the original symbols.
\end{proof}

\begin{proof}[Proof of sufficiency of \eqref{rows}--\eqref{entries}] It is enough to show that these conditions allow us to complete $P$ to an $r\times n$ latin rectangle. Then we can invoke a well-known theorem of Hall \cite{Hal45,Rys51} that states that every $r\times n$ latin rectangle can be completed to an $n\times n$ latin square . 

We break up the task of obtaining an $r\times n$ latin rectangle into two steps. 
\begin{enumerate}
\item[(i)]
We extend $P$ to an $r\times n$ partial latin rectangle by attaching $n-s$ empty columns on the right of $P$. Then we fill in some entries of the new columns by with original symbols in such a way that the resulting $r\times n$ rectangle  contains each of the original $t$ symbols $r$ times and no row or column has more than $n-t$ empty cells. We show that this procedure is always possible in the Shuffle Lemma.
\item[(ii)]
We then fill in the $r(n-t)$ empty cells with the $n-t$ new symbols to get an $r\times n$ latin rectangle. We show that this procedure is always possible in the Gap Filling Lemma. \qedhere
\end{enumerate}
\end{proof}

For example, below is a rectangle with $r=6$, $s=4$, $t=5$, and $n=7$.
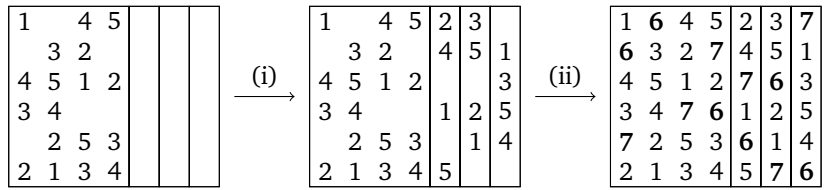
\[
\begin{tikzpicture}[scale=.4]
\begin{scope}
\draw[fill=white] (0,0) rectangle (7,6);
\draw (0,0) rectangle (4,6) rectangle (5,0) rectangle (6,6) rectangle (7,0);
\matrix 
{
1&&4&5&&&&\\
&3&2&&&&\\
4&5&1&2&&&\\
3&4&&&&&\\
&2&5&3\\
2&1&3&4\\
};
\draw[->] (7.5,3) --node[above]{(i)} (9.5,3);
\end{scope}
\begin{scope}[xshift=10cm]
\draw[fill=white] (0,0) rectangle (7,6);
\draw (0,0) rectangle (4,6) rectangle (5,0) rectangle (6,6) rectangle (7,0);
\matrix 
{
1&&4&5&2&3&\\
&3&2&&4&5&1\\
4&5&1&2&&&3\\
3&4&&&1&2&5\\
&2&5&3&&1&4\\
2&1&3&4&5\\
};
\draw[->] (7.5,3) --node[above]{(ii)} (9.5,3);
\end{scope}
\begin{scope}[xshift=20cm]
\draw[fill=white] (0,0) rectangle (7,6);
\draw (0,0) rectangle (4,6) rectangle (5,0) rectangle (6,6) rectangle (7,0);
\matrix
{
1&\textbf{6}&4&5&2&3&\textbf{7}\\
\textbf{6}&3&2&\textbf{7}&4&5&1\\
4&5&1&2&\textbf{7}&\textbf{6}&3\\
3&4&\textbf{7}&\textbf{6}&1&2&5\\
\textbf{7}&2&5&3&\textbf{6}&1&4\\
2&1&3&4&5&\textbf{7}&\textbf{6}\\
};
\end{scope}
\end{tikzpicture}
\]

\section{The Shuffle Lemma}\label{Shuffle Lemma}
Step (i)  in the proof of sufficiency of \eqref{rows}--\eqref{entries} will follow as an application of the following lemma.

\begin{lemma}[Shuffle Lemma]
Let $a,b,c$ be whole numbers. Suppose an $a\times b$ array has at least $bc$ filled cells such that no symbol occurs more than $b$ times. Then it is possible to permute the individual rows so that each column has at least $c$ entries and no symbol occurs more than once in any column.
\end{lemma}

\begin{example}
We display an example below where $a=6$, $b=3$, and $c=2$ that illustrates the workflow of the proof of the Shuffle Lemma.
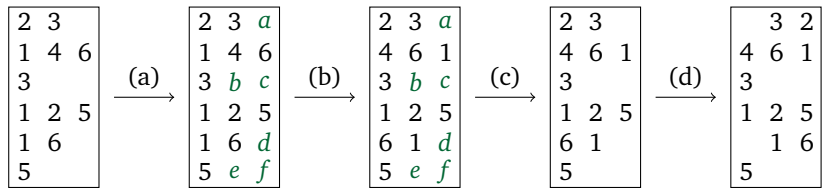
\[
\begin{tikzpicture}[scale=.4]
\begin{scope}
\draw[fill=white] (0,0) rectangle (3,6);
\matrix 
{
2&3&\\
1&4&6\\
3&&\\
1&2&5\\
1&6\\
5\\
};
\draw[->] (3.5,3) --node[above]{(a)} (5.5,3);
\end{scope}
\begin{scope}[xshift=6cm]
\draw[fill=white] (0,0) rectangle (3,6);
\matrix
{
2&3&\Dark{a}\\
1&4&6\\
3&\Dark{b}&\Dark{c}\\
1&2&5\\
1&6&\Dark{d}\\
5&\Dark{e}&\Dark{f}\\
};
\draw[->] (3.5,3) --node[above]{(b)} (5.5,3);
\end{scope}
\begin{scope}[xshift=12cm]
\draw[fill=white] (0,0) rectangle (3,6);
\matrix 
{
2&3&\Dark{a}\\
4&6&1\\
3&\Dark{b}&\Dark{c}\\
1&2&5\\
6&1&\Dark{d}\\
5&\Dark{e}&\Dark{f}\\
};
\draw[->] (3.5,3) --node[above]{(c)} (5.5,3);
\end{scope}
\begin{scope}[xshift=18cm]
\draw[fill=white] (0,0) rectangle (3,6);
\matrix 
{
2&3&\\
4&6&1\\
3&&\\
1&2&5\\
6&1&\\
5&&\\
};
\draw[->] (3.5,3) --node[above]{(d)} (5.5,3);
\end{scope}
\begin{scope}[xshift=24cm]
\draw[fill=white] (0,0) rectangle (3,6);
\matrix
{
&3&2\\
4&6&1\\
3&&\\
1&2&5\\
&1&6\\
5&&\\
};
\end{scope}
\end{tikzpicture}
\]
\end{example}

\begin{proof}
The proof will be accomplished using four steps.
\begin{enumerate}
\item[(a)] Introduce sufficiently many new "placeholder" symbols so the array is full.
\item[(b)] Permute individual rows to get column permutations.
\item[(c)] Delete the placeholder symbols.
\item[(d)] Balance the columns.
\end{enumerate} 

Steps (a) and (c) are trivial to accomplish. Thus we must show that steps (b) and (d) can always be accomplished.

First complete step (a) by adding sufficiently many distinct new symbols to $R$ so that it has $ab$ entries. To show that (b) is possible we shall show that we can get a single column $C$ by selecting one symbol from each row so that no symbol occurs more than $b-1$ times among the remaining symbols. If we can guarantee that it is always possible to obtain one column, then we can repeat our argument $b-1$ times to get the remaining columns.

To obtain our column, we shall apply Hall's Marriage Theorem. 
\begin{theorem}[Hall's Marriage Theorem]
Let $S_1,\ldots,S_n$ be sets. It is possible to select distinct representatives $s_1\in S_1,\ldots,s_n\in S_n$ if and only if for each $m\le n$, the union of $m$ sets $S_{i_1}\cup\cdots\cup S_{i_m}$ contains at least $m$ distinct elements.
\end{theorem}

Let the set $S_i$  be the $i$-th row of $R$. We need to show that the union of any $m$ distinct sets $S_{i_1},\ldots,S_{i_m}$ contains at least $m$ distinct elements. In search of a contradiction we suppose that
\begin{equation}
\label{representatives}
\left|S_{i_1}\cup\cdots\cup S_{i_m}\right|<m 
\end{equation}
for some $m\le a$. It then follows that one of the elements of $S_{i_1}\cup\cdots\cup S_{i_m}$ must occur more than $b$ times among the $bm$ symbols in the sets $S_{i_1},\ldots,S_{i_m}$. However, no element occurs in more than $b$ of the rows. We have arrived at a contradiction, so the inequality \eqref{representatives} must be false for all $m$. Therefore, by Hall's Marriage Theorem, we can obtain a set of distinct representatives, one from each row, to obtain our desired column, $C$, of distinct symbols.

Thus it is possible to choose a column of distinct symbols. However, we must also show that we can choose our column so that it contains each symbol of $R$ that occurs $b$ times in $R$.

Each symbol that occurs $b$ times in $R$ shall be called a \emph{necessary symbol}. Assume that the column $C$ is chosen to contain the maximum possible number of necessary symbols. Seeking a contradiction, we suppose that some necessary symbol $x_0$ does not occur in $C$. We shall use the notation $x\prec y$ if $y$ is a symbol in $C$ that shares a row with $x$. 

There are two possible cases. The first case is that there exists some sequence of distinct symbols $x_0\prec x_1\prec x_2\prec\cdots\prec x_k$ which terminates in a non-necessary symbol $x_k$. In this case we can replace $x_{i+1}$ in $C$ with $x_i$, which would then give us one more necessary symbol in $C$, contradicting the maximality of $C$.

The second case is that every sequence of distinct symbols $x_0\prec x_1\prec x_2\prec\cdots\prec x_k$ consists only of necessary symbols. If there are $d$ rows with a representative in $C$ that appear in at least one such a sequence, then the $d$ rows consist of $d$ necssary symbols each occuring $b$ times. Then each of the $d$ necessary symbols would already occur in $C$, contradicting the assumption that $x_0$ is not in $C$.

Now we have shown that $C$ can be chosen so that each necessary symbol occurs in $C$. Repeat this process $b-1$ times to get the columns of $C$. This completes step (b).

For step (c) we delete the dummy symbols. Now for step (d) we must show that is possible to balance the columns so that each column has at least $c$ entries. It will be sufficient to show that by permuting individual rows, any two columns can be balanced so that the number of entries in the each column is the same or differs by 1, while no symbol occurs twice in the same column. 

To show that this works in general, we may assume, without loss of generality, that the left column has more entries than the right column. Then the left column contains some symbol, $a$, that is next to an empty cell. Transfer the entry $a$ to the right column. 
\[
\begin{tikzpicture}[scale=.4]
\begin{scope}
\draw[fill=white] (0,0) grid (2,1);
\matrix 
{
a&\\
};
\draw[->] (2.5,.5) -- (3.5,.5);
\end{scope}
\begin{scope}[xshift=4cm]
\draw[fill=white] (0,0) grid (2,1);
\matrix
{
{}&a\\
};
\end{scope}
\end{tikzpicture}
\]
If $a$ occurs elsewhere in the right column, then swap the entries of the row containing it. If this places two of some symbol in the first column, then once again we swap the row that we have not yet altered. Continue on swapping rows until no column has the same symbol twice. This process must end  after a finite number of steps because each row will be swapped at most once.
\[
\begin{tikzpicture}[scale=.4]
\begin{scope}
\draw[fill=white] (0,0) grid (2,3);
\matrix
{
c&b\\
a&\\
b&a\\
};
\draw[->] (2.5,1.5) -- (3.5,1.5);
\end{scope}
\begin{scope}[xshift=4cm]
\draw[fill=white] (0,0) grid (2,3);
\matrix
{
c&b\\
&a\\
b&a\\
};
\draw[->] (2.5,.5) -- (3.5,.5);
\end{scope}
\begin{scope}[xshift=8cm]
\draw[fill=white] (0,0) grid (2,3);
\matrix
{
c&b\\
&a\\
a&b\\
};
\draw[->] (2.5,2.5) -- (3.5,2.5);
\end{scope}
\begin{scope}[xshift=12cm]
\draw[fill=white] (0,0) grid (2,3);
\matrix
{
b&c\\
&a\\
a&b\\
};
\end{scope}
\end{tikzpicture} 
\]
If the number of symbols in the two columns remains unchanged, then the process above ended by shifting a symbol from the right column to the left column. Thus the left column must have another symbol that is next to an empty cell. We can repeat this process without affecting the rows that we previously swapped. Eventually the right column must gain a new symbol because there are more rows with a single symbol in the left column than there are with a single symbol in the right column.

Thus any two columns may be balanced so that the column size differs by at most one. Since the average number of symbols per column will be at least $c$, then each column can be chosen to contain at least $c$ elements. This completes the proof.
\end{proof}

To perform step (i) of the proof of sufficiency of \eqref{rows}--\eqref{entries},  place the original symbols that do not occur in row $i$ of $P$ in row $i$ of region $A$. Then invoke the Shuffle Lemma with $a=r$, $b=n-s$, and $c=r+t-n$.

\section{The Gap Filling Lemma}\label{Gap Filling Lemma}
In this section our goal is to show that a partial latin rectangle $R$ based on $t$ symbols with $\le n-t$ empty cells in each row and column can be completed to a latin rectangle using $n-t$ new symbols.

\begin{lemma}[Gap Filling Lemma]
In a rectangular array, let $S$ be a subset of cells consisting of at most $k$ cells in any row or column, then the cells of $S$ can be filled in with $k$ symbols so that none of the $k$ symbols occurs twice in any row or column.
\end{lemma}

We can always accomplish step (ii) of the proof of sufficiency of \eqref{rows}--\eqref{entries} by appying the Gap Filling Lemma with $k=n-t$. We arrive easily at a proof of the Gap Filling Lemma by invoking a nice result of K\"{o}nig \cite[p. 128]{BR91}.

\begin{theorem}[K\"{o}nig]
If $A$ is a nonnegative integral matrix, each of whose row and column sums does not exceed the positive integer $k$, then $A$ is a sum of $k$ $(0,1)$-matrices with at most one 1 in each row and column.
\end{theorem}

\begin{proof}[Proof of Gap Filling Lemma]
Let $A$ be the $n\times n$ (0,1) matrix with 1 in the locations corresponding to the the subset $S$ and 0 in the other locations. The row and column sum of such a matrix does not exceed $k$, so by K\"{o}nig's Theorem $A$ is a sum of $k$ $(0,1)$-matrices, each with at most one 1 in each row and column:
\[
A=P_1+P_2+\cdots+P_k.
\]

We can obtain a labeling of the cells of $S$ from the nonzero entries of the matrix
\[
1\cdot P_1+2\cdot P_2+\cdots+k\cdot P_k. 
\]
\end{proof}

We note that the obvious generalization of the Gap Filling Lemma to higher dimensions fails. For example, below we display the skeleton of a 3-dimensional polyomino (shaded) and its complementary polyomino (unshaded), each of which has precisely 2 cells in each row, column, and file. 
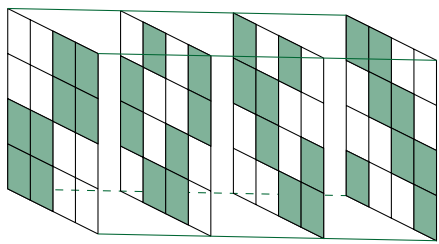
\[
\newcommand{\rec}[2]{\draw[fill=Dark!50] (#1,#2) rectangle ++(-1,-1);}
\begin{tikzpicture}[scale=.6,
every node/.style={minimum size=1cm},on grid]
\begin{scope}[every node/.append style={yslant=-.5},yslant=-.5,xscale=.5]
\draw[color=Dark,style=dashed] (0,0)-- ++(15,3.6);
  \draw (0,0) grid (4,4);
\rec{3}{4}\rec{4}{4}
\rec{3}{3}\rec{4}{3}
\rec{2}{1}\rec{2}{2}
\rec{1}{1}\rec{1}{2}
\begin{scope}[xshift=5cm,yshift=1.2cm]
  \draw (0,0) grid (4,4);
\rec{2}{4}\rec{4}{4}
\rec{1}{3}\rec{4}{3}
\rec{1}{2}\rec{3}{2}
\rec{2}{1}\rec{3}{1}
draw[fill=Dark] (1,1) rectangle ++(-1,-1);
\end{scope}
\begin{scope}[xshift=10cm,yshift=2.4cm]
  \draw (0,0) grid (4,4);
\rec{1}{4}\rec{3}{4}
\rec{1}{3}\rec{2}{3}
\rec{2}{2}\rec{4}{2}
\rec{3}{1}\rec{4}{1}
\end{scope}
\begin{scope}[xshift=15cm,yshift=3.6cm]
  \draw (0,0) grid (4,4);
\rec{1}{4}\rec{2}{4}
\rec{2}{3}\rec{3}{3}
\rec{3}{2}\rec{4}{2}
\rec{1}{1}\rec{4}{1}
\end{scope}
\draw[color=Dark] (0,4)-- ++(15,3.6);
\draw[color=Dark] (4,0)-- ++(15,3.6);
\draw[color=Dark] (4,4)-- ++(15,3.6);
\end{scope}
\end{tikzpicture}
\]
The reader is encouraged to attempt to fill in the shaded (or unshaded) cells with the symbols 0 and 1 in such a way that no symbol occurs twice in any row, column, or file. Each such attempt will result in a contradiction. Cruse noted in \cite[p. 346]{Cru74} that there was no obvious way to generalize what we have called Cruse's Theorem to higher dimensions. See \cite{MW08} for examples of partial cubes that are not completable. On the other hand, Cruse \cite{Cru74b} has shown that each partial latin hypercube of order $n$ can be embedded inside a hypercube of each order $\ge 16n^4$.

\section{Proof of Theorem~\ref{frequency}} \label{Frequency Proof}
\begin{proof}
To see the necessity of these conditions we suppose that $R$ can be completed to a latin frequency square square $S$ of type $F(n,\lambda_1,\ldots,\lambda_t)$. $S$ can be transformed to a latin square of order $n$ in the following fashion.  If we delete the symbol $i$ from $S$ then there are precisely $\mu_i$ empty cells in each row and column. By the Gap Filling Lemma we can label these cells with symbols $i_1,\ldots,i_{\lambda_i}$ such that each row and column contains exactly one symbol $i_j$ for $1\le j\le \lambda_i$. Continuing in this fashion we convert $S$ to a latin square of order $n$. 

Now consider the locations of the symbols 
\[
1_1,1_2,\ldots,1_{\mu_1},2_1,2_2,\ldots,2_{\mu_2},\ldots,k_1,k_2,\ldots,k_{\mu_k}
\]
in the $r\times s$ square. By Cruse's Theorem we know that aach row has at least $r+t-n$ entries, and each column has at least $s+t-n$ entries. Each of the $\mu_i$ symbols $i_1,\ldots,i_{\mu_i}$ must occur at least $r+s-n$ times, so the symbol $i$ occurs $\mu_i(r+s-n)$ times. Also the number of entries does not exceed
\[
\frac{rst+(n-r)(n-s)(n-t)}{n}.
\]
This shows the necessity of the conditions \eqref{F-rows}--\eqref{F-entries}.

The proof of the sufficiency of conditions \eqref{F-rows}--\eqref{F-entries} makes use of the Shuffle Lemma and the Gap Filling Lemma. We can visualize the same picture now as we did for the proof of Cruse's Theorem.
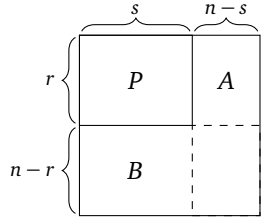
\[
\begin{tikzpicture}[scale=.3]
\draw[style=dashed] (5,4) rectangle (8,0);
\draw(0,0) rectangle (8,8);
\draw(0,8) rectangle (5,4);
\draw [decorate,decoration={brace,amplitude=5pt},yshift=5pt]
   (0.1,8)  -- (4.9,8) 
   node [black,midway,above=3pt] {\footnotesize $s$};
\draw [decorate,decoration={brace,amplitude=5pt},yshift=5pt]
   (5.1,8)  -- (7.9,8) 
   node [black,midway,above=3pt] {\footnotesize $n-s$};
\draw [decorate,decoration={brace,amplitude=5pt},xshift=-5pt]
   (0,4.1)  -- (0,7.9) 
   node [black,midway,left=4pt] {\footnotesize $r$};
\draw [decorate,decoration={brace,amplitude=5pt},xshift=-5pt]
   (0,.1)  -- (0,3.9) 
   node [black,midway,left=4pt] {\footnotesize $n-r$};
\draw (2.5,6) node{$P$};
\draw (6.5,6) node{$A$};
\draw (2.5,2) node{$B$};
\end{tikzpicture}
\]
To apply the Shuffle Lemma we place the symbol $i$ in each row of the rectangle $A$ sufficiently many times so that $i$ occurs precisely $\mu_i$ times in the $r\times n$ rectangle $P\cup A$. Now inside $A$, we relabel the symbols 
\[
1,2,\ldots,k\quad\text{as}\quad1_1,1_2,\ldots,1_{\mu_1},2_1,2_2,\ldots,2_{\mu_2},\ldots,k_1,k_2,\ldots,k_{\mu_k}
\]
in the following fashion. From left to right, top to bottom, we relabel the $j$-th occurence of $i$ as $i_{\hat{j}}$ where $\hat{j}$ is the reduction of $j\mod \mu_i$. 

We now verify the hypotheses of the Shuffle Lemma with $a=r$, $b=n-s$, and $c=r+t-n$. Condition \eqref{F-symbols} guarantees that no symbol occurs more than $b$ times, and condition \eqref{F-entries} ensures that $A$ contains at least $(n-s)(r+t-n)$ symbols. Condition \eqref{F-rows} makes it certain that each row of $A$ has at most $n-s$ symbols.

Apply the Shuffle Lemma so that each column of $A$ contains at most $n-t$ empty cells but no symbol more than once. From here we can once again relabel $i_j$ with the symbol $i$. Now the $r\times n$ rectangle formed by the $P$ and $A$ has precisely $(n-t)$ empty cells in each row and at most $(n-t)$ empty cells in each row.  By the Gap Filling Lemma, we can fill these empty cells with symbols $z_1,\ldots,z_{n-t}$ so that no $z_i$ occurs more than once in any row or column. Finally we relabel the symbols $z_1,\ldots,z_{n-t}$ as the symbols $1,\ldots, k$ so that symbol $i$ occurs precisely $\lambda_i$ times in each row. This gives us an $r\times n$ F-rectangle of type $F(n;\lambda_1,\ldots,\lambda_k)$. This F-rectangle can be completed to an $n\times n$ F-rectangle of type $F(n;\lambda_1,\ldots,\lambda_k)$ by Hall's theorem. This completes the proof.
\end{proof}

\section{Proof of Theorem~\ref{saturated}}\label{Saturated Proof}

\begin{proof} 
The proof is along the same lines as the proof we provided for Cruse's Theorem. Adjoin a $R\times (S-s)$ rectangle, $A$, to the right side of $P$. The value $g(A_j)$ will correspond to the number of original symbols in row $i$ of $A$, and $f(i)$ will correspond to the number of times the symbol $i$ occurs in $A$. 
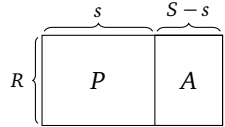
\[
\begin{tikzpicture}[scale=.3]
\draw(0,4) rectangle (8,8) rectangle (5,4);
\draw [decorate,decoration={brace,amplitude=3pt},yshift=5pt]
   (0.1,8)  -- (4.9,8) 
   node [black,midway,above=2pt] {\footnotesize $s$};
\draw [decorate,decoration={brace,amplitude=3pt},yshift=5pt]
   (5.1,8)  -- (7.9,8) 
   node [black,midway,above=2pt] {\footnotesize $S-s$};
\draw [decorate,decoration={brace,amplitude=2pt},xshift=-5pt]
   (0,4.1)  -- (0,7.9) 
   node [black,midway,left=2pt] {\footnotesize $R$};
\draw (2.5,6) node{$P$};
\draw (6.5,6) node{$A$};
\end{tikzpicture}
\]

To see the necessity of these conditions, we consider an $R\times s$ subrectangle $P$ of a saturated partial latin rectangle $Z$ of type $(R,S,T)$. Since $Z$ is saturated, and $R,S\le T$, every single cell of $Z$ must be filled. The symbols $1,\ldots,t$ will be called the original symbols, and the symbols $t+1,\ldots,T$ will be called the new symbols.

To see the necessity of \eqref{Sat-rows} and \eqref{Sat-columns} we note that at most $T-t$ new symbols can fill the empty cells of any row or column, so there must be at least $s-(T-t)$ entries in each row of $P$ and at least $R-(T-t)$ entries in each column of $P$.

For \eqref{Sat-symbols} we observe that $A$ has $S-s$ columns, so no original symbol can occur more than $S-s$ times in $A$ (condition \eqref{Sat-a}).  At most $S-s$ original symbols can occur in each row (condition \eqref{Sat-b}). Also the region $A$ must contain at least $(S-s)(R+t-T)$ original symbols (condition \eqref{Sat-d}). 

To see that condition \eqref{Sat-c} holds we let $P(i)$ be the number of original symbols in  $A$. Each row must have $S$ symbols, at most $(T-t)$ of which are new symbols. Let $P(i)$ be the number of original symbols in row $i$. Then number of original symbols in row $i$ will be $P(i)+g(A_i)$. The inequality $P(i)+g(A_i)\ge S-(T-t)$ implies
\[
g(A_i)\ge S-(T-t)-P(i)=S-(T-t)-(t-|A_i|)=S-T+|A_i|.
\]
Thus conditions \eqref{Sat-rows}--\eqref{Sat-symbols} are necessary.

Now we assume that conditions \eqref{Sat-rows}--\eqref{Sat-symbols} hold  for some rectangle $P$ and show that $P$ can be extended to a saturated $(R,S,T)$-partial latin rectangle. Consider the $r\times(S-s)$ rectangle $A$ that has its rows given by the set of $f,g$ representatives of $\mathcal{A}$. We shall invoke the Shuffle Lemma with $a=r$, $b=S-s$, and $c=R+t-T$. 

We first show that the Shuffle Lemma can be applied. Condition \eqref{Sat-a} guarantees that no symbol occurs more than $b$ times. Condition \eqref{Sat-b} ensures that the symbols fit inside an $a\times b$ array. Condition \eqref{Sat-d} ensures that at least $bc$ entries are in $A$.

We use the Shuffle Lemma to make $A$ a partial latin rectangle such that each column has at most $T-t$ gaps. Conditions \eqref{Sat-rows}, \eqref{Sat-columns}, and \eqref{Sat-c} then ensure that the $R\times S$ rectangle $P\cup A$ also has at most $T-t$ empty cells in any row or column. Now the Gap Filling Lemma ensures that the remaining cells can be filled with the $T-t$ new symbols.
\end{proof}

\section{Quasiembeddings of Latin Squares}\label{k-plex}
\subsection{Cruse's Theorem for Partial Latin Squares}

Since latin squares are more often encountered than latin rectangles, it is worth noting a formulation of the theorem in the special case of extending a partial latin square of order  $n$ to a latin square of order $n+k$.

\begin{corollary}\label{Cruse Corollary}
A partial latin square, $P$, of order $n$ can be extended to a latin square of order $n+k$ if and only $P$ can be completed (using the original $n$ symbols) to a partial latin square, $P'$, of order $k$ such that following two conditions hold.
\begin{enumerate}
\item Each row, column, and symbol is represented at least $n-k$ times in $P'$.
\item $P'$ has at least $k(n-k)$ empty cells.
\end{enumerate}
\end{corollary}

\begin{example}
Consider the following partial latin squares of order $5$.
\[
\begin{tikzpicture}[scale=.4]
\begin{scope}
\draw[fill=white] (0,0) rectangle (5,5);
\draw (0,0) grid (5,5);
\matrix
{
&1&2&3&4\\
1&&3&4&5\\
2&3&&5&1\\
3&4&5&&2\\
4&5&1&2&\\
};
\draw (2.5,0) node[below] {$A$};
\end{scope}
\begin{scope}[xshift=7cm]
\draw[fill=white] (0,0) rectangle (5,5);
\draw (0,0) grid (5,5);
\matrix
{
4&1&2&3&\\
1&2&3&4&\\
2&3&4&1&\\
3&4&1&2&\\
{}&&&&\\
};
\draw (2.5,0) node[below] {$B$}; 
\end{scope}
\begin{scope}[xshift=14cm]
\draw[fill=white] (0,0) rectangle (5,5);
\draw (0,0) grid (5,5);
\matrix
{
5&&&1&2\\
&3&4&&1\\
&4&3&5&\\
1&&5&2&4\\
2&1&&3&5\\
};
\draw (2.5,0) node[below] {$C$};
\end{scope}
\end{tikzpicture}
\]

By the theorem, the square $A$ can be embedded in latin squares of order $n$ for $n =5, 6$ or $n\ge9$.

The square $B$ can be embedded in a latin square of order $n$ for each $n\ge8$ (a result that we could have obtained from Evan's theorem).

The square $C$ can be embedded in a latin square of order $n$ for each $n\ge7$.
\end{example}

\subsection{Quasi-embeddings and Brualdi's Conjecture}
A \emph{transversal} of a latin square is a subset of $n$ cells such that each of the $n$ rows, columns, and symbols is represented in one of the cells. A \emph{partial transversal of order $m$} is a subset of $m$ cells such that each of the $n$ rows, columns, and symbols is represented in at most one of the cells.

A \emph{$k$-plex} in a latin square of order $n$ is a subset of $kn$ cells such that each of $n$ symbols occurs exactly $k$ times \cite{Wan02}. A $k$-plex is a generalization of the notion of a transversal because each transversal is a $1$-plex. A \emph{partial $k$-plex of order $m$} in a latin square is a subset of $m$ cells such that each row and each column contain at most $k$ entries, and each symbol occurs at most $k$ times.

Each partial transversal of size $m$ is a partial $1$-plex of order $m$. Deleting any $k(n-m)$ entries from a $k$-plex results in a partial $k$-plex of size $m$. 

The intersection number of two latin squares $L$ and $K$ is the number of cells where the $(i,j)$ entry of $L$ equals the $(i,j)$ entry of $K$. The references \cite{DM09} and \cite{How10} investigated the spectrum $I(n,n+k)$ of intersection numbers for squares of orders $n$ and $n+k$. When $k\ge n$, $I(n,n+k)=[0,n^2]$ by Evan's Theorem.  When $k\le n$, the upper bound for members of $I(n,n+k)$ is $n^2-k(n-k)$. We say that a latin square $L$ of order $n$ is \emph{quasi-embedded} inside a latin square $K$ of order $n+k$ if all but $k(n-k)$ cells of $L$ occur in the upper left corner of $K$.

\begin{theorem}
A latin square, $L$, of order $n$ can be quasi-embedded inside a square of order $n+k$ if and only if $L$ contains a partial $k$-plex of order $k(n-k)$.
\end{theorem}
\begin{proof}
Both statements are equivalent to the existence of a partial latin square $P$ inside $L$ such that $P$ contains $n^2-k(n-k)$ cells with each row, column, and symbol represented at least $n-k$ times in $P$. The partial latin square $L\setminus P$ is a partial $k$-plex of order $k(n-k)$ and Corollary~\ref{Cruse Corollary} guarantees $P$ can be completed to a square of order $n+k$.
\end{proof}

\begin{conjecture}
Every latin square of order $n$ has a partial $k$-plex of order $k(n-k)$ for each $k\le n$.
\end{conjecture}

When $k=1$, this conjecture is precisely Brualdi's Conjecture. However, the $k=1$ case seems to be the most difficult in general. At the other extreme, a partial $(n-1)$-plex of order $n-1$ can be found by selecting any $n-1$ cells, so the conjecture is true when $k=n-1$.

\providecommand{\bysame}{\leavevmode\hbox to3em{\hrulefill}\thinspace}

\end{document}